\documentclass[12pt,a4paper]{article}
\usepackage[a4paper,top=2cm,bottom=2cm,left=2.2cm,right=2.2cm]{geometry}
\usepackage{amsmath}
 \usepackage{graphicx}
\usepackage{amsthm}
\usepackage{color}
\usepackage{fancyhdr}
\usepackage{pdfpages}
\usepackage{amssymb}
\usepackage{paralist}
\usepackage[colorlinks=true,citecolor=blue,linkcolor=blue,anchorcolor=blue,urlcolor=blue]{hyperref}
\allowdisplaybreaks
\newtheorem{theorem}{Theorem}[section]

\newtheorem{lemma}{Lemma}[section]

\newtheorem{remark}{Remark}[section]
\newtheorem{definition}{Definition}[section]

\numberwithin{equation}{section}
\allowdisplaybreaks
 
\begin{document}
\title{ Nontrivial solutions for  semilinear elliptic systems  via Orlicz-Sobolev theory}
\author{Fei Fang\footnote{Corresponding author. E-mail address: fangfei68@163.com}
\\School of Mathematical Sciences, Peking  University \\Beijing, 100871,  China}
 \date{}
 \maketitle
\noindent \textbf{\textbf{Abstrct:}}
In this paper, the semilinear elliptic systems with Dirichlet boundary value are considered
\begin{align}
\left\{
\begin{array}{ll}
  -\Delta v=f(u) & \mathrm{in}\ \Omega,  \\
 -\Delta u=g(v) & \mathrm{in}\ \Omega,   \\
 u=0,  \ v=0  & \mathrm{on}\ \partial\Omega,
\end{array}
\right.
\end{align}
 We extend the notion of subcritical growth from polynomial growth to N-function
growth. Under N-function growth, nontrivial solutions are obtained via Orlicz-Sobolev spaces
and variational methods. It's also noteworthy that the nonlinear term $g(v)$ does not have to  satisfy the usual  Ambrosetti-Rabinowitz condition. So,
in a sense, we enrich recent results of  D. ~G. de~Figueiredo, J. ~M.  do~{\'O} and  B. ~Ruf  [D. ~G.  de~Figueiredo,        J. ~M.  do~{\'O},        B. ~Ruf,        An {O}rlicz-space approach to superlinear elliptic systems,        J.  Funct.  Anal.  224 (2005) 471--496].

\noindent \textbf{\textbf{Keywords:}}   Semilinear elliptic systems, Orlicz-sobolev spaces,   concentration-compactness principle

 \section{Introduction}
In this paper,   we shall be concerned with the existence of nontrivial solutions of semilinear Hamilton systems
\begin{align}
(OP)\ \ \ \left\{
\begin{array}{ll}
  -\Delta v=f(u) & \mathrm{in}\ \Omega,  \\
 -\Delta u=g(v) & \mathrm{in}\ \Omega,   \\
 u=0,  \ v=0  & \mathrm{on}\ \partial\Omega,
\end{array}
\right.
\end{align}
where $\Omega$ is a bounded domain in  $R^N(N\geq 1)$ with smooth boundary $\partial \Omega$,   and $f,   g:\Omega\rightarrow R$ are continuous functions.
Let $g$ be an odd and inverse  function and  $p(t)=g^{-1}(t)$. Then from $ -\Delta u=g(v)$, we can solve for  $v$  and  plug it into $-\Delta v=f(u)$.    So we can transfer the semilinear Hamilton systems into the following problem
\begin{align}
(NP)\ \ \ \left\{
\begin{array}{ll}
  \Delta(p(\Delta u))=f(u) & \mathrm{in}\ \Omega,  \\
  u=\Delta u=0 & \mathrm{in}\ \partial\Omega.
\end{array}
\right.
\end{align}

As in \cite{MR2763349,  r4,   r12,   r7,   r1},   let
 $$P(t):=\int_0^t p(s)ds,  \ \ \ \tilde{P}(t)=\int_0^t p^{-1}(s)ds,  $$
then  $P$ and   $\tilde{P}$  are complementary $N$-functions(see \cite{r2,     MR1890178,r38}).

In order to  equip with  an  Orlicz-Sobolev space for the operator $\Delta(p(\Delta \cdot))$,  we make some assumptions on the function  $p$
\begin{compactitem}
  \item[\quad\quad($p_0$):] $p(t)\in C^1(0, +\infty),   p(t)>0, p'(t)>0$  for $t>0$,
  \item[\quad\quad($p_1$):] $1<p^{-}:=\displaystyle \inf _{t>0}\frac{tp(t)}{P(t)} \leq p^{+}:=\displaystyle\sup_{t>0}\frac{tp(t)}{P(t)}<+\infty$ ,
  \item[\quad\quad($p_2$):] $0<a^{-}:=\displaystyle \inf _{t>0}\frac{tp^{'}(t)}{p(t)} \leq a^{+}:=\displaystyle\sup_{t>0}\frac{tp^{'}(t)}{p(t)} <+\infty$ .
\end{compactitem}
From ($p_1$),    $P(t)$ satisfies the $\Delta_2$ condition,   i.e., there exists a  constant $k>0$ such that
$$P(2t)\leq kP(t),   \quad t>0.$$
Under the conditions ($p_0$) and ($p_1$),   the Orlicz space  $L^{P}$ coincides with the set (equivalence classes)
of measurable functions  such that $u: \Omega\rightarrow R$
\begin{equation}  \label{ge4}
\int_{\Omega}P(|u|)dx<+\infty.
\end{equation}
The space $L^{P}$ is a Banach space endowed with the Luxemburg norm
\begin{equation*}
|u|_{P}:=\mathrm{inf}\left\{ k>0,  \int_{\Omega} P\left(\frac{|u|}{k}\right)%
dx<1\right\}.
\end{equation*}

The  Orlicz-Sobolev space  $W^{m,P}(\Omega)$  consists of those
(equivalence classes of ) functions $u$ in  $L^{P}(\Omega)$ whose distributional derivatives $D^{\alpha} u$ also belong to $L^{P}(\Omega)$
for all $\alpha$ with $\alpha\leq m$. It may be checked by the same method used for ordinary Sobolev spaces  $W^{m,P}(\Omega)$ is a
 Banach space with respect to the norm
\begin{equation*}
\|u\|_{W^{m,  P}(\Omega)}:=\max_{0\leq |\alpha|\leq m}\left|| D^{\alpha}u|\right|_{P}.
\end{equation*}

In this paper, we shall use second order Orlicz-Sobolev space $W^{2,P}(\Omega)$  to describe problem (NP).
And   $W^{2,P}(\Omega)$ is endowed with   the following  norm
\begin{equation*}
\|u\|_{W^{2,  P}(\Omega)}:=\left||\Delta u|\right|_{P}.
\end{equation*}
As in the case of ordinary Sobolev spaces, $W_{0}^{2,   P}(\Omega)$ is taken to be the closure
of c  $C_{0}^{\infty}$ in  $W^{2,   P}(\Omega)$.

Many properties of Orlicz-Sobolev spaces are obtained by very straightforward
generalization ofthe proofs ofthe same properties for ordinary Sobolev spaces.   In past two decades, Orlicz-Sobolev theory was widely applied in nonlinear differential equations (see \cite{n3, n7, MR2146049, n2, n1, n4}  and references therein). The reader is referred to \cite{r2,MR1890178,r38} for more details on Orlicz-Sobolev spaces theory. In the proofs of our results we shall use
the following results..
\begin{lemma}[See \protect\cite{r2,   MR1890178,  r38}]
\label{l21} Under the conditions $(p_0)$, $(p_1)$,   the spaces  $L^{P}(\Omega)$,   $W_{0}^{1,  P}(\Omega)$,   $W_{0}^{2,  P}(\Omega)$,   $W^{2,  P}(\Omega)$   and  $W^{1,  P}(\Omega)$ are separable and reflexive Banach spaces..
\end{lemma}
Now we define a   sequence of $N$-functions $Q_0,   Q_1,   Q_2\cdots $ as follows:
\begin{align}
Q_0&=P(t),\nonumber\\
(Q_k)^{-1}&=\int_0^t\frac{(Q_k)^{-1}(\tau)}{\tau^{\frac{N+1}{N}}}d\tau<\infty,   \ k=1,  2,  \cdots.
\end{align}
For each $k$,   we assume that
\begin{equation}\label{g67}
   \int_0^t\frac{(Q_k)^{-1}(\tau)}{\tau^{\frac{N+1}{N}}}d\tau<\infty,
\end{equation}
replacing $Q_k$,  if necessary, with another $N$-function equivalent to it near infinity
and satisfying the above formula.

Let $J=J(P)$ be the smallest nonnegative integer such that
\begin{equation}
   \int_1^{\infty}\frac{(Q_J)^{-1}(\tau)}{\tau^{\frac{N+1}{N}}}d\tau<\infty.
\end{equation}
\begin{definition}\label{gd1}
   Let $M(t)>0$ denote the class of positive, continuous, increasing functions of $t>0$. If  $\mu\in M$,   the space $C_{\mu}(\bar{\Omega})$ consisting of those functions for which the norm
   $$\|u\|_{C_{\mu}(\bar{\Omega})}=\|u\|_{C(\bar{\Omega})}+\sup_{x,  y\in \Omega,   x\not=y}\frac{|u(x)-u(y)|}{\mu(|x-y|)}.  $$
   is finite
\end{definition}
It is easily to show that $C_{\mu}(\bar{\Omega})$ is a Banach space under the above norm.
\begin{lemma}[See  \cite{r2,     MR1890178,r38}]\label{gr2}
Let  $P$ and  $Q$ are   $N$-functions.
\begin{compactitem}
  \item[\quad\emph{(1)}] If $2\geq J(P)$,   then  $W^{2,  P}(\Omega)\rightarrow L^{Q_2}$ .   Moreover,  if $Q$ is an $N$-function increasing essentially more slowly than $Q_2$ near infinity,  then the imbedding  $W^{2,  P}(\Omega)\rightarrow L^{Q}$ exists and is compact.
  \item[\quad\emph{(2)}] If $2> (P)$,   then $W^{2,  P}(\Omega)\rightarrow C_Q^0(\Omega)=C^0\cap L^{\infty}(\Omega)$.
  \item[\quad\emph{(3)}] If $2>J(P)$  and $\Omega$ satisfies the strong local Lipschitz condition,  then $W^{2, P}(\Omega)\rightarrow C_{\mu}^{1-J}(\bar{\Omega})$,  where
  $$\mu(t)=\int_{t^{-n}}^{\infty}\frac{(Q_J)^{-1}(\tau)}{\tau^{\frac{N+1}{N}}}d\tau.  $$
 Moreover,  the imbedding $W^{2,   P}(\Omega)\rightarrow C^{1-J}(\bar{\Omega})$ is compact and so is $W^{2, P}(\Omega)\rightarrow C_{\nu}^{2-J-1}(\bar{\Omega})$ provided $\nu\in M$ and $\mu/\nu\in M$.

\end{compactitem}
\end{lemma}
Now we consider the spaces $E:=W^{2,  P}(\Omega)\cap W_0^{1,  P}(\Omega)$ and the following functional
\begin{equation}
I(u)=\int_{\Omega}P(\Delta u)dx-\int_{\Omega} F(u)dx:=\mathcal{P}(u)-\mathcal{F}(u).
\end{equation}
As in \cite{r4},  we easily know that the critical points of $I$  are just the weak solutions of problem (NP).
\section{Main results and its proof}
\begin{compactitem}
  \item[($f_\ast$):]  There exists an odd increasing homeomorphism $h: R\rightarrow R$ and nonnegative constants $a_{1},   a_{2}$ such that
$$|f(t)|\leq a_{1}+a_{2}h(|t|),    \forall t\in R,  $$
and  $$\lim_{t\rightarrow+\infty}\frac{H(t)}{Q_2(kt)}=0,  \ \ \forall k>0,  $$  where
\begin{equation} \label{ge7}
H(t):=\int_{0}^{t} h(s)ds.
\end{equation}
\end{compactitem}

\begin{compactitem}
  \item[$ (f_1):$] There exist two constants  $\theta>p^+$ and $R_0>0$ such that  ,
\begin{equation}
0<\theta F(t)\leq tf(t), \ \  t\not=0,  |t|>R_0 \nonumber
\end{equation}
\item[$ (f_2):$] $f(t)=o(p(t))$ as $t\rightarrow 0$.
\end{compactitem}
Similar to the condition  ($p$),   for $H$,  we assume that the following condition£º
\begin{compactitem}
\item[($h$):] $1<h^{-}:=\displaystyle \inf_{t>0}\frac{th(t)}{H(t)} \leq h^{+}:=\displaystyle
\sup_{t>0}\frac{th(t)}{H(t)}<+\infty$.
\end{compactitem}
\textbf{For convenience of statement,   we will denote $Q_2(t)$ by  $P_{\ast}(t)$}.   Moreover, we   assume that $p_{\ast}^{-}$,   $h^{+}$ and  $h^{-}$ satisfies following inequality
\begin{equation}
p^{+}<h^{-}\leq h^{+}<p_{\ast}^{-}.
\end{equation}

\begin{theorem}\label{gt621}
Assume that the conditions ($p_0$),   ($p_1$),   ($p_2$),   ($f_{\ast}$),   ($f_{1}$) and ($f_{2}$) are satisfied,  then problem (OP) or  problem (NP)
has at least one nontrivial solution.
\end{theorem}
Before the proof of Theorem \ref{gt621}, we need prove some useful lemmas.
\begin{lemma}[\cite{r12}]\label{l36}
Let $\rho(u)=\int_{\Omega}P(u)dx$, we have
\begin{compactitem}
  \item[\quad\emph{(1)}:] if $|u|_P<1$, then $|u|_P^{p^{+}}\leq\rho(u)\leq |u|_P^{p^{-}}$;
  \item[\quad\emph{(2)}:] if $|u|_P>1$, then $|u|_P^{p^{-}}\leq\rho(u)\leq |u|_P^{p^{+}}$;
   \item[\quad\emph{(3)}] if $0<t<1$, then $t^{p^{+}}P (u)\leq P(tu)\leq t^{p^{-}}P(u)$;
  \item[\quad\emph{(4)}] if $t>1$, then $t^{p^{-}}P(u)\leq P(tu)\leq t^{p^{+}}P(u)$.
\end{compactitem}
\end{lemma}
 Similar to Lemma \ref{l36},  we have
\begin{lemma}\label{l4}
\begin{compactitem}
  \item[\quad \emph{(1)}:] If $\|u\|<1$, then
$\|u\|^{p^{+}}\leq \mathcal{P}(u)\leq \|u\|^{p^{-}}$.
  \item[\quad \emph{(2)}:] If $\|u\|>1$, then
$\|u\|^{p^{-}}\leq \mathcal{P}(u)\leq \|u\|^{p^{+}}$.
\end{compactitem}
\end{lemma}

\begin{lemma}\label{gl622}
  The functional $\mathcal{P}\in C(E,   R)$ is convex, sequentially weakly lower semi-continuous   and
  $$\mathcal{P}'(u)v=\int_{\Omega}p(\Delta u)\Delta v,  \ \forall u,  v\in E.$$  Moreover, the mapping $\mathcal{P}'(u): E\rightarrow E^{\ast}$ is a strictly monotone, bounded homeomorphism, and is of $S^{+}$ type, namely
  $$u_n\rightharpoonup u\ \mbox{\ and\ }\limsup \mathcal{P}'(u_n)(u_n-u)\leq 0\ \mbox{imply}\ u_n\rightarrow u.  $$
\end{lemma}

\begin{lemma}\label{dl31}
Under the condition  ($f_{\ast}^1$),   the functionals $\mathcal{F}(u): E\rightarrow R$ is sequentially weakly continuous,
$\mathcal{F}(u)\in C^1(X,   R)$,   and for all $u,   \phi \in E$,
$$\mathcal{F}^{'}(u)\phi=\int_{\Omega}f(u)\phi dx.  $$
 The mapping $\mathcal{F}^{'}: E\rightarrow E^{\ast} $ is sequentially weakly-strongly continuous,   namely,
\begin{equation} u_n\rightharpoonup u {\ \mbox{implies}\ } \mathcal{F}^{'}(u_n)\rightarrow
\mathcal{F}^{'}(u),
\end{equation}
where $\rightharpoonup$ and $\rightarrow$ denote the weak convergence and strong convergence in $E$ respectively.
\end{lemma}
In fact, the proof of Lemma \ref{gl622} and Lemma \ref{dl31} is standard. So we take the similar methods in \cite{r4} and     \cite{r7} to prove   Lemma \ref{gl622} and Lemma \ref{dl31}, respectively.

\renewcommand{\proofname}

\begin{proof}[\textbf{The proof of Theorem \emph{\ref{gt621}}}]
\textbf{Step 1},  The functional  $I(u)$ satisfies (P.S.) condition.

Let $\{u_n\}$  be a  (P.S.) sequence,  namely
$$|I(u)|\leq M\  \mbox{and}\ I'(u_n)\rightarrow 0.  $$
Next we will prove that $\{u_n\}$  is bounded in $E$. Using the condition $ (f_1)$, we have
\begin{align}\label{g611}
M+\frac{1}{\theta} o(1) \|u_n\|_{E}&\geq I(u_n)-\frac{1}{\theta}I'(u_n)u_n\nonumber\\
&=\int_{\Omega}P(\Delta u_n)-\int_{\Omega} F(u_n)dx-\frac{1}{\theta}\left(\int_{\Omega}p(\Delta u_n)\Delta u_ndx-\int_{\Omega}f(u_n)u_ndx\right)\nonumber\\
&\geq \left(1-\frac{p^{+}}{\theta}\right)\int_{\Omega}P(\Delta u_n)dx+\int_{|t|>R_0}\left(\frac{1}{\theta}f(u_n)u_n-F(x,   u_n)\right)dx\nonumber\\
&\ \ +\int_{|t|\leq R_0}\left(\frac{1}{\theta}f(u_n)u_n-F(x,   u_n)\right)dx  \nonumber\\
&\geq \left(1-\frac{p^{+}}{\theta}\right)\int_{\Omega}P(\Delta u_n)dx-M_1.
\end{align}
So,
$$C(1+o(1)\|u_n\|_{E})\geq \int_{\Omega}P(\Delta u_n)dx.$$
This   implies that $\{u_n\}$ is bounded in  $E$, then $$u_n\rightharpoonup u\ \mbox{in}\ E,$$
$$u_n\rightarrow u\ \mbox{in}\  L^{H},  $$
moreover, $$\langle\mathcal{P}'(u_n),   u_n-u\rangle\rightarrow 0.  $$
Since $I(u)=\mathcal{P}(u)-\mathcal{F}(u)$,    one has
$$\langle I'(u_n),   u_n-u\rangle=\langle\mathcal{P}'(u_n),   u_n-u\rangle+\int_{\Omega} f(u_n)(u_n-u)dx.  $$
By the condition  ($f_\ast$) and H\"{o}lder inequalities in Orlicz spaces,  we obtain
\begin{align}
 \int_{\Omega} |f(u_n)(u_n-u)|dx&\leq a_1\int_{\Omega}|u_n-u|dx+a_2\int_{\Omega}h(|u_n|)(u_n-u)\nonumber\\
 &\leq (a_1\|1\|_{\tilde{H}}+a_2\|h(u_n)\|_{\tilde{H}})\|u_n-u\|_{H}.
\end{align}
Recalling that $u_n\rightarrow u$ in $L^{H}$, we get
$$\int_{\Omega} |f(u_n)(u_n-u)|dx\rightarrow  0.$$
Consequently, $$\limsup \mathcal{P}'(u_n)(u_n-u)\leq 0.  $$
Now from  Lemma \ref{gl622}, we  easily know that $u_n\rightarrow u$ in $E$.

\textbf{Step 2},   The functional $I(u)$  has mountain geometry,   that is

\begin{compactitem}
  \item [\quad\quad(1)] There exist two constants  $R_0,   r$ such that if $\|u\|=R$,  then $I(u)> r$;
  \item [\quad\quad(2)] There exists $u_0\in E$ such that if $\|u_0\|\geq R$,  then $I(u)<r$.
\end{compactitem}
By the conditions  ($f_{\ast}$) and ($f_2$),   for all $\varepsilon>0$,   There is  a constant $C=C(\varepsilon)>0$ such that
$$|F(t)|\leq a_1\varepsilon P(t)+Ca_2H(t).$$
Therefore,
\begin{align}
I(u)&=\int_{\Omega}P(\Delta u)dx-\int_{\Omega} F(u)dx\nonumber\\
    &\geq \int_{\Omega}P(\Delta u)dx-a_1\varepsilon \int_{\Omega}P(u)dx-Ca_2\int_{\Omega} H(u)dx\nonumber\\
    &\geq \|u\|^{p^{+}} -a_1\varepsilon |u|^{p^{-}}-Ca_2|u|^{h^{-}}\nonumber\\
    &\geq \|u\|^{p^{+}} -C_1\varepsilon \|u\|^{p^{-}}-C_2\|u\|^{h^{-}}(\mbox{by Sobolev embedding theorem}).
\end{align}
From the assumption $h^{-}>p^{+}$,   for $\|u\|_{H}>0 $  and  $\varepsilon$ small enough,  we immediately
obtain (1).

For  (2),   by the condition ($f_1$),   there exist two constants $C, c_0>0$ such that
$$F(t) \geq  C|t|^{\theta}-c_0.$$
Thus,
\begin{align}
I(tu_0)&=\int_{\omega} P(t\Delta u_0)dx-\int_{\Omega} F(tu_0)dx\nonumber\\
&\leq t^{p^{+}}\int_{\Omega}P(\Delta u_0)dx-t^{\theta}\int_{|u_0|\geq R_0} C|u_0|dx +(c_0+M)|\Omega|,
\end{align}
where,   $M=\sup\{|F(t)|: |t|\leq R_0 \}$.   Since $\theta>p^{+}$,  we may choose $u_0$ and $R_0>0$ such that $|\{x\in \Omega: |u_0(x)|>t_0\}|>0$,   hen we
have
$I(tu_0)\rightarrow -\infty$ as  $t\rightarrow \infty$.    So applying  the mountain pass theorem,  $I(u)$  has at least one nontrivial solution.
\end{proof}

\begin{theorem}\label{gt622}
Under the conditions of Theorem \emph{\ref{gt621}},  if $f(t)$ is an odd function,  then there exist infinitely many nontrivial weak solutions to problem (OP) or problem (NP).
\end{theorem}

\begin{proof}[\textbf{Proof of Theorem \emph{\ref{gt622}}}]
Since $f(t)$ is an odd function,  the functional is  even.   So we will use the \lq\lq $Z_2$-symmetric\rq\rq  version of the Mountain Pass Theorem
(see \cite{r88} to accomplish the proof of Theorem \ref{gt622}.   By this Theorem,  we  only  need verify  the following condition
\begin{compactitem}
  \item [(3)] For  arbitrary finite dimensional subspace $E_1\subset E$,    the set $S=\{u\in E_1: I(u)\geq0\}$ is bounded in $E$.
\end{compactitem}
In fact,
\begin{align}
I(u)&=\int_{\Omega} P(\Delta u)dx-\int_{\Omega} F(u)dx\nonumber\\
    &\leq \int_{\Omega}P(\Delta u)dx-\int_{\Omega} C|u|^{\theta}dx+c_0M.
\end{align}
It is well know that  all norms are equivalent in arbitrary finite dimensional subspace.   Therefore, $(\int_{\Omega} |u|^{\theta}dx)^{\frac{1}{\theta}}$  can  be regarded as the norm of $E_1$.    Now the relation  $\theta>p^{+}$ implies that  $S$  is a bounded set.
\end{proof}
Now we   replace  $f(t)$  in the conditions ($f_{\ast}$),   ($f_1$) and ($f_2$) by a new function $m(t)$,   then we can obtain three conditions, which  be denoted by  ($m_{\ast}$),   ($m_1$) and ($m_2$),  respectively.
\begin{theorem}\label{gt625}
    When  $f(t)=P'_{\ast}(t)+\lambda m(t)$,   assume that $m(t)$ satisfies the conditions $(m_{\ast})$,   $(m_1)$ and $(m_2)$,  then there is a constant $\lambda_0$ which depends on   $P(t),  N,   \theta$  such that if $\lambda>\lambda_0$,   problem  (OP)  or problem  (NP) has at least one nontrivial solution.
\end{theorem}

Similar to \cite{r12},   we can adopt similar methods to prove the following  the  concentration-compactness principle
\begin{lemma}\label{gl625}
Let ,  $u_n\rightharpoonup u$ in $E$  and  $P_{\ast}(|u_n|)\rightharpoonup \nu$ and $P(|\Delta u_n|)\rightharpoonup \mu$ in $\mathcal{U}(R^N)$. Then there exist an at most countable set $J$, a family  $\{x_j\}_{j\in J}$  of distinct points in $R^N$ and a family $\{\nu_{j}\},  \{\mu_{j}\}_{j\in J}$ such that
\begin{compactitem}
  \item [\quad\emph{($a$)}] $\nu=P_{\ast}(|u|)+\sum_{j\in J}\nu_j\delta_{x_j}$;
   \item [\quad\emph{($b$)}] $\mu\geq P(|\Delta u|)+\sum_{j\in J}\mu_j\delta_{x_j}$;
    \item [\quad\emph{($c$)}] $0<\nu_j\leq \max\left\{S_0^{p_{\ast}^{-}}\mu_j^{\frac{p_{\ast}^{-}}{p_{-}}},   S_0^{p_{\ast}^{+}}\mu_j^{\frac{p_{\ast}^{+}}{p_{-}}},   S_0^{p_{\ast}^{-}}\mu_j^{\frac{p_{\ast}^{-}}{p_{+}}},   S_0^{p_{\ast}^{+}}\mu_j^{\frac{p_{\ast}^{+}}{p_{+}}}\right\}$;
\end{compactitem}
where,  $\mathcal{U}(R^N)$ denotes the space of Radon space,  $\mu,   \nu$ are negative measure in $\mathcal{U}(R^N)$,  $S_0$ denotes optimal constant of Orlicz-Sobolev embedding,   $\delta_{x_j}$  is the Dirac measure of mass 1 concentrated at $x_j$.
\end{lemma}

\begin{lemma}\label{gl626}
Assume that $g(t)$ satisfies the conditions ($g_{\ast}$),  ($g_1$) and ($g_2$). Let   $\{u_n\}\subset H$  be a (PS) sequence of functional $I$ with energy level  $c$.   If
\begin{align}
c<c_0:=\min&\left\{\left(\frac{p^{-}_{\ast}}{p^{+}}-1\right)\left[\frac{p^{-}}{p_{\ast}^{-}}\left(\frac{1}{S_0^{p_{\ast}^{-}}}\right)^
{\frac{p^{-}}{p_{\ast}^{-}}}\right]^{\frac{p_{\ast}^-}{p_{\ast}^--p^{-}}},  \left(\frac{p^{-}_{\ast}}{p^{+}}-1\right)\left[\frac{p^{-}}{p_{\ast}^{-}}\left(\frac{1}{S_0^{p_{\ast}^{+}}}\right)^
{\frac{p^{-}}{p_{\ast}^{+}}}\right]^{\frac{p_{\ast}^+}{p_{\ast}^+-p^{-}}},  \right.  \nonumber\\
 &\ \ \left.  \left(\frac{p^{-}_{\ast}}{p^{+}}-1\right)\left[\frac{p^{-}}{p_{\ast}^{-}}\left(\frac{1}{S_0^{p_{\ast}^{-}}}
 \right)^{\frac{p^{+}}{p_{\ast}^{-}}}\right]^{\frac{p_{\ast}^-}{p_{\ast}^--p^{+}}},  \left(\frac{p^{-}_{\ast}}{p^{+}}-1\right)\left[\frac{p^{-}}{p_{\ast}^{-}}\left(\frac{1}{S_0^{p_{\ast}^{+}}}\right)^
{\frac{p^{+}}{p_{\ast}^{+}}}\right]^{\frac{p_{\ast}^+}{p_{\ast}^+-p^{+}}}\right\},
\end{align}
then $\{u_n\}$ has a convergent subsequence in $E$.
\end{lemma}
\begin{proof}[\textbf{Proof of Lemma \emph{\ref{gl626}}}]
Since $\{u_n\}$ is a (PS) sequence,  similar to the proof of theorem \ref{gt621}, we easily know that $\{u_n\}$ is bounded in $E$.  $E$ is a reflexive Banach
space,   so there exists a  weakly convergent subsequence of $\{u_n\}$ which we denote by $\{u_n\}$.   Thus
$$u_n\rightharpoonup u\ \mbox{in} E,$$
$$u_n\rightarrow u \ \mbox{in}\  L^{Q}(\Omega),$$
where,  $Q$ is a $N$-function increasing essentially more
slowly than $P_{\ast}(t)$ near infinity.   Applying the above concentration-compactness
principle, we infer
\begin{compactitem}
  \item [\quad($a$)] $\nu=P_{\ast}(|u|)+\sum_{j\in J}\nu_j\delta_{x_j}$;
   \item [\quad($b$)] $\mu\geq P(|\Delta u|)+\sum_{j\in J}\mu_j\delta_{x_j}$;
    \item [\quad($c$)] $0<\nu_j\leq \max\left\{S_0^{p_{\ast}^{-}}\mu_j^{\frac{p_{\ast}^{-}}{p_{-}}},   S_0^{p_{\ast}^{+}}\mu_j^{\frac{p_{\ast}^{+}}{p_{-}}},   S_0^{p_{\ast}^{-}}\mu_j^{\frac{p_{\ast}^{-}}{p_{+}}},   S_0^{p_{\ast}^{+}}\mu_j^{\frac{p_{\ast}^{+}}{p_{+}}}\right\}$.
\end{compactitem}

Let $\phi\in C^{\infty}(R^N)$ be such that $\phi=1,   x\in B(x_k,   \varepsilon)$,   $\phi=0,   x\in B(x_k,   2\varepsilon)^{c}$,    $|\nabla \phi|\leq \frac{2C}{\varepsilon^{p^{-}p^{+}}}$ and
$|\Delta \phi|\leq \frac{2C}{\varepsilon^{2p^{-}p^{+}}}$,   where $x_i\in\bar{\Omega}$ belongs to the support of $\nu$.
Since the sequence $\{u_n\}$ is bounded in $E$,
$$0\leftarrow \langle I'(u_n),   u_n\phi\rangle=\int_{\Omega}p(\Delta u_n)\Delta(u_n\phi)dx-\int_{\Omega}P'_{\ast}(u_n)u_n\phi dx-\lambda\int_{\Omega}m(u_n)\phi dx,  $$
that is,
$$\int_{\Omega}P'_{\ast}(u_n)u_n\phi dx+\lambda\int_{\Omega}m(u_n)\phi dx-\int_{\Omega}p(\Delta u_n)\Delta(u_n\phi)dx\rightarrow 0,\ n\rightarrow+\infty.$$
By the H\"{o}lder inequalities in Orlicz spaces and Lemma \ref{l36},
\begin{align}\label{ere}
  0&\leq \lim_{n\rightarrow\infty} \left|\int_{\Omega}p(\Delta u_n)\nabla u_n\cdot\nabla \phi dx\right|\nonumber\\
   &\leq \lim_{n\rightarrow\infty}\|p(\Delta u_n)\|_{\tilde{P}}\|\nabla  u_n\nabla\cdot \phi\|_{P}\nonumber\\
   &\leq C\|\nabla  u\cdot\nabla \phi\|_{P}\nonumber\\
   &\leq\left\{ \begin{array}{l}
              C\left(\int_{\Omega} P(\nabla u \cdot\nabla\phi)dx\right)^{\frac{1}{p^{-}}} \\
             C\left(\int_{\Omega} P(\nabla u \cdot\nabla\phi)dx\right)^{\frac{1}{p^{+}}}
           \end{array}\right.  \nonumber\\
 &\leq\left\{\begin{array}{l}
            C\left(\int_{B(x_k,   2\varepsilon)\cap\Omega} P(|\nabla u||\nabla\phi|)dx\right)^{\frac{1}{p^{-}}} \\
           C\left(\int_{B(x_k,   2\varepsilon)\cap\Omega}  P(|\nabla u| |\nabla\phi|)dx\right)^{\frac{1}{p^{+}}}
         \end{array}\right.  \nonumber\\
     &\leq \left\{\begin{array}{ll}
                    C\left(\int_{B(x_k,   2\varepsilon)\cap\Omega} \frac{1}{\varepsilon^{p^{-}}}P(|\nabla u|)dx\right)^{\frac{1}{p^{-}}} \\
                   C\left(\int_{B(x_k,   2\varepsilon)\cap\Omega} \frac{1}{\varepsilon^{p^{-}}}P(|\nabla u|)dx\right)^{\frac{1}{p^{+}}}
                 \end{array}\right.  \nonumber\\
    &\leq\left\{\begin{array}{l}
                  C\left(\int_{B(x_k,   2\varepsilon)\cap\Omega} \frac{1}{\varepsilon^{\frac{N}{2}}}dx\right)^{\frac{2}{N}} \left(\int_{B(x_k,   2\varepsilon)\cap\Omega}(P(|\nabla u|))^{\frac{N}{N-2p^{-}}}dx\right)^{\frac{N-2p^{-}}{Np^{-}}}\\
                  C\left(\int_{B(x_k,   2\varepsilon)\cap\Omega} \frac{1}{\varepsilon^{\frac{N}{2}}}dx\right)^{\frac{2}{N}} \left(\int_{B(x_k,   2\varepsilon)\cap\Omega}(P(|\nabla u|))^{\frac{N}{N-2p^{-}}}dx\right)^{\frac{N-2p^{-}}{Np^{+}}}
                \end{array}\right.  \nonumber\\
    &\leq \left\{\begin{array}{l}
                    C\left(\int_{B(x_k,   2\varepsilon)\cap\Omega}(P(|\nabla u|))^{\frac{N}{N-2p^{-}}}dx\right)^{\frac{N-2p^{-}}{Np^{-}}} \rightarrow 0 \  \mbox{as}\ \varepsilon\rightarrow 0\\
                   C\left(\int_{B(x_k,   2\varepsilon)\cap\Omega}(P(|\nabla u|))^{\frac{N}{N-2p^{-}}}dx\right)^{\frac{N-2p^{-}}{Np^{+}}}\rightarrow 0 \ \mbox{as}\ \varepsilon\rightarrow 0,
                 \end{array}\right.
   \end{align}
where last equation of \eqref{ere} converges to zero because of the property absolute continuity of integral.
Note that when the property absolute continuity of integral is used, one need to check the integrand is integrable. In fact,  $(P(|\nabla t|))^{\frac{N}{N-2p^{-}}}$ increase  essentially more slowly than  $P_{\ast}(t)$,  so $(P(|\nabla u|))^{\frac{N}{N-2p^{-}}}$ is integrable. In the same way,  we also obatin
\begin{align}
 0&\leq \lim_{n\rightarrow\infty} \left|\int_{\Omega}u_np(\Delta u_n)\Delta \phi dx\right|\rightarrow 0, \  \mbox{as}\varepsilon\rightarrow 0.
\end{align}
Thus,
\begin{align}
0&=\lim_{\varepsilon\rightarrow0}\lim_{n\rightarrow+\infty}\left[\int_{\Omega}P'_{\ast}(u_n)u_n\phi dx+\lambda\int_{\Omega}m(u_n)\phi dx-\int_{\Omega}p(\Delta u_n)\Delta(u_n\phi)dx\right]\nonumber\\
&\leq  \lim_{\varepsilon\rightarrow0}\lim_{n\rightarrow+\infty}\left[\int_{\Omega}p_{\ast}^{+}P_{\ast}(u_n)\phi dx+\lambda\int_{\Omega}m(u_n)\phi dx-\int_{\Omega}p(\Delta u_n)\Delta(u_n\phi)dx\right]\nonumber\\
&=p_{\ast}^{+}\nu_{k} -\lim_{\varepsilon\rightarrow0}\lim_{n\rightarrow+\infty}\left[\lambda\int_{\Omega}m(u_n)\phi dx-\int_{\Omega}p(\Delta u_n)\Delta u_n\phi dx\right.  \\\nonumber
&\ \ \left.  -2\int_{\Omega}p(\Delta u_n)(\nabla u_n\cdot \nabla \phi)dx-\int_{\Omega}u_np(\Delta u_n)\Delta\phi dx\right]\nonumber\\
&=p_{\ast}^{+}\nu_{k}-\lim_{\varepsilon\rightarrow0}\lim_{n\rightarrow+\infty}\left[\lambda\int_{\Omega}m(u_n)\phi dx-\int_{\Omega}p(\Delta u_n)\Delta u_n\phi dx\right]\nonumber\\
&\leq p_{\ast}^{+}\nu_{k}-\lim_{\varepsilon\rightarrow0}\lim_{n\rightarrow+\infty}\left[\lambda\int_{\Omega}m(u_n)\phi dx-p^{-}\int_{\Omega}P(\Delta u_n)\phi dx\right]\nonumber\\
&\leq p_{\ast}^{+}\nu_{k}-p^{-}\mu_k.
\end{align}
This implies that  $ p_{\ast}^{+}\nu_{k}\geq p^{-}\mu_k.  $

\item [(1)] If $\max\left\{S_0^{p_{\ast}^{-}}\mu_j^{\frac{p_{\ast}^{-}}{p_{-}}},   S_0^{p_{\ast}^{+}}\mu_j^{\frac{p_{\ast}^{+}}{p_{-}}},   S_0^{p_{\ast}^{-}}\mu_j^{\frac{p_{\ast}^{-}}{p_{+}}},   S_0^{p_{\ast}^{+}}\mu_j^{\frac{p_{\ast}^{+}}{p_{+}}}\right\}=S_0^{p_{\ast}^{-}}\mu_j^{\frac{p_{\ast}^{-}}{p_{-}}}$,
then $$\left(\frac{\nu_k}{S_0^{p_{\ast}^-}}\right)^{\frac{p^{-}}{p_{\ast}^-}}p^{-}\leq P_{\ast}^{-}\nu_k.  $$
So,   either $v_k=0$, or
$$v_k\geq \left[\frac{p^{-}}{p_{\ast}^{-}}\left(\frac{1}{S_0^{P_{\ast}^{-}}}\right)^{\frac{p^{-}}{p_{\ast}^{-}}}\right]^{\frac{p_{\ast}^-}{p_{\ast}^--p^{-}}}.  $$
Now we prove that $\nu_k=0$.   And if not,  for some $k$, we get $v_k\geq \left[\frac{p^{-}}{p_{\ast}^{-}}\left(\frac{1}{S_0^{P_{\ast}^{-}}}\right)\right]^{\frac{p_{\ast}^-}{p_{\ast}^-}-p^{-}},$  moreover, for a (PS) sequence $\{u_n\}$, we deduce
\begin{align}\
c&=\lim_{n\rightarrow+\infty}I(u_n)=\lim_{n\rightarrow+\infty}\left(I(u_n)-\frac{1}{p^{+}}I'(u_n)u_n\right)\nonumber\\
&=\lim_{n\rightarrow+\infty}\left(\int_{\Omega}P(\Delta u_n)dx-\int_{\Omega} P^{\ast}(u_n)dx-\lambda \int_{\Omega}M(u_n)dx\right)\nonumber\\
&\ \ -\lim_{n\rightarrow+\infty}\left(\frac{1}{p^{+}}\int_{\Omega}p(\Delta u_n)\Delta u_n dx-\frac{1}{p^{+}}\int_{\Omega} P'_{\ast}(u_n)u_n dx- \frac{\lambda }{p^{+}}\int_{\Omega}m(u_n)u_ndx\right)\nonumber\\
&\geq \lim_{n\rightarrow}\left(\frac{p^{-}_{\ast}}{p^{+}}-1\right)\int_{\Omega} P_{\ast}(u_n) dx+\lim_{n\rightarrow +\infty}\int_{\Omega}\left[\frac{1}{p^{+}}m(u_n)u_n-M(u_n)\right]dx\nonumber\\
&\geq\left(\frac{p^{-}_{\ast}}{p^{+}}-1\right)\int_{\Omega} d\nu= \sum_{i=1}^{k}    \nu_{i}  \geq \nu_k\nonumber\\
&\geq \left(\frac{p^{-}_{\ast}}{p^{+}}-1\right)\left[\frac{p^{-}}{p_{\ast}^{-}}\left(\frac{1}{S_0^{p_{\ast}^{-}}}\right)^
{\frac{p^{-}}{p_{\ast}^{-}}}\right]^{\frac{p_{\ast}^-}{p_{\ast}^--p^{-}}}.
\end{align}

\item [(2)] If $\max\left\{S_0^{p_{\ast}^{-}}\mu_j^{\frac{p_{\ast}^{-}}{p_{-}}},   S_0^{p_{\ast}^{+}}\mu_j^{\frac{p_{\ast}^{+}}{p_{-}}},   S_0^{p_{\ast}^{-}}\mu_j^{\frac{p_{\ast}^{-}}{p_{+}}},   S_0^{p_{\ast}^{+}}\mu_j^{\frac{p_{\ast}^{+}}{p_{+}}}\right\}=S_0^{p_{\ast}^{+}}\mu_j^{\frac{p_{\ast}^{+}}{p_{-}}}$,   similar to calculation of (1),   we have
    $$c>\left(\frac{p^{-}_{\ast}}{p^{+}}-1\right)\left[\frac{p^{-}}{p_{\ast}^{-}}\left(\frac{1}{S_0^{p_{\ast}^{+}}}\right)^
{\frac{p^{-}}{p_{\ast}^{+}}}\right]^{\frac{p_{\ast}^+}{p_{\ast}^+-p^{-}}}.  $$

\item [(3)] If $\max\left\{S_0^{p_{\ast}^{-}}\mu_j^{\frac{p_{\ast}^{-}}{p_{-}}},   S_0^{p_{\ast}^{+}}\mu_j^{\frac{p_{\ast}^{+}}{p_{-}}},   S_0^{p_{\ast}^{-}}\mu_j^{\frac{p_{\ast}^{-}}{p_{+}}},   S_0^{p_{\ast}^{+}}\mu_j^{\frac{p_{\ast}^{+}}{p_{+}}}\right\}=S_0^{p_{\ast}^{-}}\mu_j^{\frac{p_{\ast}^{-}}{p_{+}}}$,   similar to calculation of (1),   we have
$$c\geq\left(\frac{p^{-}_{\ast}}{p^{+}}-1\right)\left[\frac{p^{-}}{p_{\ast}^{-}}\left(\frac{1}{S_0^{p_{\ast}^{-}}}\right)^
{\frac{p^{+}}{p_{\ast}^{-}}}\right]^{\frac{p_{\ast}^-}{p_{\ast}^--p^{+}}}.  $$
\item [(4)] If $\max\left\{S_0^{p_{\ast}^{-}}\mu_j^{\frac{p_{\ast}^{-}}{p_{-}}},   S_0^{p_{\ast}^{+}}\mu_j^{\frac{p_{\ast}^{+}}{p_{-}}},   S_0^{p_{\ast}^{-}}\mu_j^{\frac{p_{\ast}^{-}}{p_{+}}},   S_0^{p_{\ast}^{+}}\mu_j^{\frac{p_{\ast}^{+}}{p_{+}}}\right\}=S_0^{p_{\ast}^{+}}\mu_j^{\frac{p_{\ast}^{+}}{p_{+}}}$,   similar to calculation of (1),   we have
    $$c>\left(\frac{p^{-}_{\ast}}{p^{+}}-1\right)\left[\frac{p^{-}}{p_{\ast}^{-}}\left(\frac{1}{S_0^{p_{\ast}^{+}}}\right)^
{\frac{p^{+}}{p_{\ast}^{+}}}\right]^{\frac{p_{\ast}^+}{p_{\ast}^+-p^{+}}}.  $$

The above four case is contradictory with assumptions of Lemma \ref{gl626},   so $\{u_n\}$ has a convergent subsequence.
\end{proof}

\begin{proof}[\textbf{Proof of Theorem \emph{\ref{gt625}}}]
Next we will use mountain pass theorem to prove Theorem \ref{gt625}.   Therefore,  we need to check the following conditions
\item[\quad\quad(1)] there are   two constants $R,   r$  such that if $\|u\|=R$, then $I(u)>r$;
\item[\quad\quad(2)] there is a $u_0\in E$ such that if $\|u_0\|>R$, then $I(u_0)<r$;
\item[\quad\quad(3)] there is a continuous cuve $\gamma:[0,  1]\rightarrow H$ such that $\gamma0)=0,   \gamma(1)=u_0$  and $\sup_{0\leq t\leq 1}I(\gamma (t))\leq c_0$.
We can take the methods of  Theorem \ref{gt621}   to check the conditions (1)and (2),   so we omit it.   Next,  we will focus on  verification of  the condition (3).

For each $u\in E$,  we deduce
\begin{align}
I(u)&=\int_{\Omega} P(\Delta u)dx-\int_{\Omega}P_{\ast}(u)-\lambda\int_{\Omega} M(u)dx\nonumber\\
    &\leq\max\left\{\begin{array}{c}
                      \|u\|^{p^{+}} -\|u\|^{p_{\ast}^-}_{P_{\ast}}-\lambda\int_{\Omega}M(u)\\
                      \|u\|^{p^{-}} -\|u\|^{p_{\ast}^+}_{P_{\ast}}-\lambda\int_{\Omega}M(u) \\
                       \|u\|^{p^{+}} -\|u\|^{p_{\ast}^-}_{P_{\ast}}-\lambda\int_{\Omega}M(u) \\
                       \|u\|^{p^{-}} -\|u\|^{p_{\ast}^-}_{P_{\ast}}-\lambda\int_{\Omega}M(u)
                    \end{array}\right.  \nonumber\\
                    &\leq \max\left\{\begin{array}{c}
                      \|u\|^{p^{+}} -\|u\|^{p_{\ast}^-}_{P_{\ast}}-\lambda C\|u\|_{L^{\theta}}^{\theta}\\
                      \|u\|^{p^{-}} -\|u\|^{p_{\ast}^+}_{P_{\ast}}-\lambda C\|u\|_{L^{\theta}}^{\theta} \\
                       \|u\|^{p^{+}} -\|u\|^{p_{\ast}^-}_{P_{\ast}}-\lambda C\|u\|_{L^{\theta}}^{\theta} \\
                       \|u\|^{p^{-}} -\|u\|^{p_{\ast}^-}_{P_{\ast}}-\lambda C\|u\|_{L^{\theta}}^{\theta}
                    \end{array}.  \right.
\end{align}
Set
\begin{equation}
\Psi(u):=\|u\|^{p^{+}} -\|u\|^{p_{\ast}^-}_{P_{\ast}}-\lambda\int_{\Omega}M(u):=\max\left\{\begin{array}{c}
                      \|u\|^{p^{+}} -\|u\|^{p_{\ast}^-}_{P_{\ast}}-\lambda C\|u\|_{L^{\theta}}^{\theta}\\
                      \|u\|^{p^{-}} -\|u\|^{p_{\ast}^+}_{P_{\ast}}-\lambda C\|u\|_{L^{\theta}}^{\theta} \\
                       \|u\|^{p^{+}} -\|u\|^{p_{\ast}^-}_{P_{\ast}}-\lambda C\|u\|_{L^{\theta}}^{\theta} \\
                       \|u\|^{p^{-}} -\|u\|^{p_{\ast}^-}_{P_{\ast}}-\lambda C\|u\|_{L^{\theta}}^{\theta}
                    \end{array}.  \right.
\end{equation}
Let $\omega\in E$,   $\|\omega\|_{P_{\ast}}=1$ and $\Phi(t)=\Psi(t\omega)$.   Obviously, $\lim_{t\rightarrow +\infty}\Phi(t)=-\infty$.    This implies that $\Phi(t)$
 has a maximum value, that is,   there exists a constant $t_{\lambda}>0$ such that $\sup_{t>0}\Phi(t)=\Phi_{t_{\lambda}}$.  So,
$$0=\Phi'(t_{\lambda})=p^{+}t_{\lambda}^{p^{+}-1}\|\omega\|^{p^{+}}-p_{\ast}^-t_{\lambda}^{p_{\ast}^--1}-C\lambda \theta t_{\lambda}^{\theta-1}\|\omega\|_{L^{\theta}}^{\theta}.  $$
This guarantees that
\begin{equation}\label{ge1234}
p^{+}\|\omega\|^{p^{+}-1}=p_{\ast}^-t_{\lambda}^{p_{\ast}^--p^{+}}+C\lambda \theta t_{\lambda}^{\theta-p^{+}}\|\omega\|_{L^{\theta}}^{\theta},
\end{equation}
so $t_{\lambda}$ is bounded.   Since $\lambda \theta \|\omega\|_{L^{\theta}}^{\theta}\rightarrow +\infty$ as $\lambda\rightarrow +\infty$,   by virtue of \eqref{ge1234}, we know that
$\lim_{\lambda\rightarrow +\infty}t_{\lambda}=0$,   moreover, $\lim_{\lambda\rightarrow +\infty}\Psi(t\omega)=0$.   Therefore, there exists a $\lambda_0>0$ such that if $\lambda>\lambda_0$,  then
$$\sup_{t\geq 0}I(u)\leq \sup_{t\geq 0}\Psi(u) <c_0.$$
\end{proof}

\begin{remark}\label{qq}
  In \cite{MR2146049}, using  the H\"{o}lder inequalities in Orlicz  spaces,  D.G. de Figueiredo and his coauthors
  assume that the nonlinear terms $f$ and $g$ satisfy the $N$-function growth and construct the \lq\lq tilde-map\rq\rq
  such that  problem (OP) turns into  a strongly indefinite problem. And then they use a minimax theorem  and  an approximate method  of   finite dimension (Galerkin approximation) to obtain a nontrivial solution.  In \cite{MR2146049}, the  usual  Ambrosetti-Rabinowitz conditions are assumed for problem (P), but in our main results, the nonlinear term $g$ does not have to satisfy the usual  Ambrosetti-Rabinowitz condition. Moreover, we also  extend the notion of subcritical growth from polynomial growth to $N$-function growth for problem (OP).  So, in   a sense, we enrich  their results.

\end{remark}

 \end{document}